\documentclass[12pt,a4paper]{article}
\usepackage[top=1in, bottom=1in, left=1in, right=1in]{geometry}
\usepackage[utf8]{inputenc}
\usepackage{amsmath}
\usepackage{textcomp}
\usepackage{amssymb}
\usepackage{diagbox}
\usepackage{ytableau}
\renewenvironment{bmatrix}{\begin{ytableau}}{\end{ytableau}}
\usepackage{youngtab}
\usepackage{graphicx,epstopdf}
\usepackage{ mathrsfs }
\usepackage{graphics}
\usepackage{stmaryrd}
\usepackage{amsfonts}
\usepackage{hyperref}
\usepackage{amsthm}
\usepackage[skip=2pt]{caption}
\usepackage[figurename=Fig.]{caption}
 \usepackage[usenames,dvipsnames]{pstricks}
 \usepackage{epsfig}
 \usepackage{pst-grad} 
 \usepackage{pst-plot} 
 \usepackage{algorithm}
 \usepackage[noblocks]{authblk}
\usepackage{wasysym}
\usepackage[utf8]{inputenc}

\newtheorem{thm}{Theorem}[section]
\theoremstyle{definition}

\newtheorem{con}{Conjecture}[section]
\begin{document}

\title{\textbf{On Total Coloring of Some Classes of Regular Graphs}}

\author[1]{S. Prajnanaswaroopa, J. Geetha and K.Somasundaram}
\author[2]{Hung Lin-Fu}
\author[3]{Narayanan. N}

\affil[1]{Department of Mathematics, Amrita School of Engineering-Coimbatore\\ Amrita Vishwa Vidyapeetham, India.

sntrm4@rediffmail.com,\{s\_prajnanaswaroopa, j\_geetha, s\_ sundaram\}@cb.amrita.edu}
\affil[2]{Department of Applied Mathematics, National Chiao Tung University, Hsinchu 30010, Taiwan. 

hlfu@math.nctu.edu.tw}
\affil[3]{Department of Mathematics, Indian Institute of Technology Madras, Chennai, India. 

naru@iitm.ac.in}
\date{}
\maketitle

\maketitle
\noindent\textbf{Abstract:} In this paper, we have obtained the total chromatic number of some classes of Cayley graphs, odd graphs and mock threshold graphs. \\

\maketitle

\section{Introduction}


\indent All the graphs considered here are finite, simple and undirected.\\
Let $G=(V(G),E(G))$ be a graph with the sets of vertices $V(G)$ and edges $E(G)$ respectively. A \textit {total coloring} of \textit{G} is a mapping $ f:V(G)\cup E(G) \rightarrow C $, where $C$ is a set of colors, satisfying the following three conditions (a)-(c):\\

 (a) $ f(u)\neq f( v)$ for any two adjacent vertices $ u, v\in V(G) $

 (b) $ f( e)\neq f( e')$ for any two adjacent edges $ e,  e' \in E(G)$ and

 (c) $ f( v)\neq f( e)$ for any vertex $ v\in V(G)$ and any edge $ e\in E(G)$ incident to \textit{v}.\\

The \textit{total chromatic number}  of a graph \textit{G}, denoted by $ \chi''(G) $, is the minimum number of colors that suffice in a total coloring. It is clear that $ \chi''(G) \geq \Delta+1 $, where $\Delta$ is the maximum degree of \textit{G}.  Behzad \cite{BEZ} and Vizing \cite{VGV} independently conjectured (Total Coloring Conjecture (TCC)) that  for every graph $G$, $\chi''(G)\leq \Delta+2 $.  The total coloring conjecture is a long standing conjecture and has defied several attempts in a complete proof. It is also proved that the decidability algorithm for total coloring is NP-complete even for cubic bipartite graph \cite{SAN}. But still, a lot of progress has been made in proving the TCC. It is easily seen to be true for complete graphs, bipartite, complete multipartite graphs. It was also showed to be true for all graphs with degree $\Delta\le5$ and $\Delta\ge n-5$, where $n$ is the number of vertices, using techniques like enlarge-matching argument and fan recoloring process \cite{YAP}. For planar graphs, TCC is proved for all $\Delta\neq6$ using discharging and charging methods. The total coloring conjecture has also been proved for several other classes of graphs.  Good survey of techniques and other results on total coloring can be found in Yap \cite{YAP}, Borodin \cite{OVB} and Geetha et al. \cite{GNS}
 \section{Total Colorings of Cayley Graphs}
   Cayley graphs are those whose vertices are the elements of groups and adjacency relations are defined by subsets of the
    groups.  Let $\Gamma$ be a multiplicative group with identity 1. For $S\subseteq \Gamma, 1\notin S \text{ and }
    S^{-1}=\{s^{-1}:s\in S\}=S$ the \textit{Cayley Graph} $X=Cay(\Gamma, S)$ is the undirected graph having vertex set
    $V(X)=\Gamma$ and edge set $E(X)=\{(a,b): ab^{-1}\in S\}$.
 The Cayley graph associated with $\Gamma=\mathbb{Z}_n$, the group of integers modulo $n$ under addition operation, is called a \textit{circulant graph}. Note that the \textit{powers of cycle graph} is a circulant graph with the generating set $S=\{1,2,\ldots,k,n-k,\ldots,n-2,n-1\}.$ All Cayley graphs are vertex transitive.\\
  
 In other words,  given a sequence of positive integers $1\leq d_1<d_2<...<d_l\leq  \lfloor \frac{n}{2} \rfloor$, the circulant graph
    $G=C_n(d_1,d_2,...,d_l)$ has vertex set $V=Z_n=\{0,1,2,...,n-1\}$, two vertices $x$ and $y$ being adjacent iff $x=(y\pm
    d_i) \mod n$ for some $i, 1\leq i\leq l$ and a graph is a power of cycle, denoted $C^k_n$, $n$ and $k$ are integers, $1\leq k<\lfloor\frac{n}{2}\rfloor$, if
    $V(C^k_n)=\{v_0,v_1,...,v_{n-1}\}$ and $E(C^k_n)=E^1 \cup E^2 \cup ...\cup E^k$, where
    $E^i=\{e_0^i,e_1^i,...,e^i_{n-1}\}$ and $e_j^i=(v_j,v_{(j+i) \mod \ n })$ and $0 \leq j \leq n-1$, and $1\leq i\leq k$.

    Campos and de Mello \cite{CADM2} proved that $C_n^2, n\neq 7$, is type-I and $C_7^2$ is type-II. They \cite{CADM1} 
    verified the TCC for power of cycle $C_n^k, n$ even and
    $2<k<\frac{n}{2}$ and also showed that one  can obtain a
    $\Delta(C_n^k)+2$-total coloring for these graphs in polynomial time. They also proved that $C_n^k$ with $n\cong 0 \mod (\Delta(C_n^k)+1)$ are
    type-I and they proposed the following conjecture.  
    \begin{con}
    Let $G=C_n^k$, with $2\leq k< \lfloor \frac{n}{2}
        \rfloor$. Then,

        $\chi''(G)=\begin{cases}\Delta(G)+2, & \text{ if }k>\frac{n}{3}-1    \text{ and } n  \text{ is \ odd}\\ \Delta(G)+1,  &
        \text{ otherwise}.  
    \end{cases}$

    \end{con}
Recently, Zorzi et al. \cite{ZOR} proved the conjecture for certain classes of power of cycle graphs. In particular, for even $k$ and large $n$ the conjecture is true.

 A latin square is an $n\times n$ array consisting of $n$ entries of numbers (or symbols) with each row and column containing only one instance of each element. This means the rows and columns are permutations of one single $n$ vector with distinct entries. A latin square is said to be commutative if it is symmetric.  A latin square containing numbers is said to be idempotent if each diagonal element contains the number equal to its row (column) number. In addition, if the rows of the latin square are just cyclic permutations (one-shift of the elements to the right) of the previous row, then the latin square is said to be circulant (anti-circulant, if the cyclic permutations are actually left shifts), the matrix (corresponds to the latin square) can be generated from a single row vector.  The latin square

\begin{center}

\begin{tabular}{|c|c|c|c|c|c|c|}\hline1&k+2&2&k+3&\ldots&2k+1&k+1\\\hline k+2&2&k+3&3&\ldots&k+1&1\\\hline\ldots&\ldots&\ldots&\ldots&\ldots&\ldots&\ldots\\\hline\ldots&\ldots&\ldots&\ldots&\ldots&\ldots&\ldots\\\hline k+1&1&k+2&2&\ldots&k&2k+1\\
\hline
\end{tabular}

\end{center}

\noindent is    anti-circulant, commutative and idempotent.  
The entries of the square are as follows:
\[ L=(l_{ij})=\begin{cases} m,\ \text{ if }\ i+j=2m\\k+1+m, \ \text{otherwise.}\end{cases}\]

From the above, it can be easily seen that the latin square corresponding to the matrix $L$ is commutative, idempotent and also anti-circulant.

 
 Campos and de Mello \cite{CADM2}, proved the TCC for the powers of cycles of even order. In the following theorems,  using latin squares we prove  some classes of powers of cycles  of even order are Type 1. 
 
 \begin{thm}
 Let $G$ be a power of cycle graph $C_n^k$ with $k=\frac{n-2}{4}+k'$ with $n,k,k'$ being non-negative integers with $n\geq4$, $n$ even and $gcd(n,x)=1$,   $\frac{n-2}{4}<x<k'$. Then, $G$ is a type I graph.
 \end{thm}
 \begin{proof}
 The adjacency matrix of a power of cycle graph $C_n^k$ (or, in fact, any circulant graph) is a symmetric circulant matrix $C=(c_{ij})$ with $1$'s when $i,j$ differ by $s$, where $s\in S$, the generating set of the Cayley graph $C_n^k$. For example, the adjacency matrix of $C_{10}^2$ is
 
 $A(C_{10}^2) = $
 $\begin{bmatrix}0&1&1&0&0&0&0&0&1&1\\
 1&0&1&1&0&0&0&0&0&1\\
 1&1&0&1&1&0&0&0&0&0\\
 0&1&1&0&1&1&0&0&0&0\\
 0&0&1&1&0&1&1&0&0&0\\
 0&0&0&1&1&0&1&1&0&0\\
 0&0&0&0&1&1&0&1&1&0\\
 0&0&0&0&0&1&1&0&1&1\\
 1&0&0&0&0&0&1&1&0&1\\
 1&1&0&0&0&0&0&1&1&0\end{bmatrix}$.\\
 
We know that $\Delta(C_n^k)=2k$. For giving a total $2k+1$ coloring of $C_n^k$ in the case $k=\frac{n-2}{4}$, where $n,k$ are non-negative integers, $n$ is even, we form the color matrix (a matrix which gives the color of the vertices in the diagonals and the color corresponding to edges in the other entries) by first filling the non-zero entries and diagonal entries in the first $(k+1)\times n$ submatrix of the color matrix with the corresponding entries of the first $k+1$ rows of the latin square. The first non-zero entry of the $k+2$-th row of the color matrix is determined by the $k+2$-th entry of the $2$-nd row of the color matrix (as the color matrix is symmetric). The next non zero entries of the $k+2$-th row are determined by the cyclic order of the previous  rows. Similarly we determine the non-zero entries of remaining  rows (the first entry determined by the symmetric property of the color matrix and the next entries determined by the cyclic order of the previous rows). Thus  continuing, we can fill all the entries of the color matrix satisfying the total coloring conditions, giving us a $2k+1$ total coloring. For example, the $(5\times5)$ anti-circulant, commutative and idempotent latin square is

 \begin{center}
\begin{tabular}{|c|c|c|c|c|}\hline1&4&2&5&3\\\hline 4&2&5&3&1\\\hline2&5&3&1&4\\\hline5&3&1&4&2\\\hline3&1&4&2&5 \\
\hline
\end{tabular}
\end{center}
Now, this corresponds to $k=2$ (as $2k+1=5$). Thus, the filled color matrix for $C_{10}^2$ (where $n=10=2(2(2)+1)=2(2k+1)$) is:\\
\begin{center}
 $\begin{bmatrix}1&4&2&0&0&0&0&0&5&3\\
 4&2&5&3&0&0&0&0&0&1\\
 2&5&3&1&4&0&0&0&0&0\\
 0&3&1&4&2&5&0&0&0&0\\
 0&0&4&2&5&3&1&0&0&0\\
 0&0&0&5&3&1&4&2&0&0\\
 0&0&0&0&1&4&2&5&3&0\\
 0&0&0&0&0&2&5&3&1&4\\
 5&0&0&0&0&0&3&1&4&2\\
 3&1&0&0&0&0&0&4&2&5\end{bmatrix}$
 \end{center}
 
\noindent which is seen to a $2k+1=5$-total coloring. Note that $0$-s do not constitute colors in the above color matrix. 
\\
 
 
 
 
 
 We see that the powers of cycles graph $C_n^{k}$, $k=\frac{n-2}{4}+k'$ is a disjoint union of $C_n^{\frac{n-2}{4}}$ and a circulant graph.  When $\frac{n-2}{4}<k<\frac{n}{2}$ the edge disjoint graph that is added to $C_n^{\frac{n-2}{4}}$ is a class I graph (as the edge disjoint graph added is a circulant graph, which is edge colorable by $\Delta$ colors, if the generating set of the graph is also a generating set of the group, which is guaranteed if $gcd(n,x)=1$, where $\frac{n-2}{4}<x<k'$ ; the graph is edge colorable with $\Delta$ colors, where $\Delta$ be the maximum degree of the edge disjoint graph added \cite{STO}). Now, since in a type I total coloring of $C_n^{\frac{n-2}{4}}$, we have given $\frac{n}{2}$ colors to the vertices and  $\chi(C_n^k)\leq \frac{n}{2}$ (since the vertices can always be arranged into independent sets as $[0,\frac{n}{2}],[1,\frac{n}{2}+1],\ldots,[k,n-1]$ provided $k<\frac{n}{2}$), we need to only give a coloring to the edges of the remaining (added) circulant graph, which is seen to require only $\Delta$ extra colors. Thus the total coloring of the graph $C_n^k$ is again seen to require only $2\left(\frac{n-2}{4}\right)+1+\Delta=2\left(\frac{n-2}{4}\right)+2k'+1=2k+1$  colors.
 
 \end{proof}
\begin{thm}
Let $G$ be a powers of cycle graph $C_n^k$ with $n=s(2m+1)$, with $s$ being even and $2m+1-i=k\,,1\le i\le k+1$. Then the graph $C_n^k$ is total colorable with $2k+1$ colors.
\end{thm}
\begin{proof}
We observe that  $n\equiv0 \bmod (k+i)$ with $1\le i\le k+1$. We also know that there exists a commutative idempotent latin square of odd order, $k+i$ in this case, which we call $C'$. Now, we consider two tableau of the form\\
Tableau $B'$:\\
\\
\begin{tabular}{|c|c|c|c|c|}\hline
$2m+2$&&&&\\\hline$2m+3$&$2m+2$&&&\\\hline$2m+4$&$2m+3$&$2m+2$&&\\\hline\ldots&\ldots&\ldots&$2m+2$&\\\hline$2k+1$&\ldots&\ldots&\ldots&$2m+2$\\\hline
\end{tabular}\\
\\

Tableau $A'$:\\
\\
\begin{tabular}{|c|c|c|c|c|}\hline
$2k+1$&$2k$&$2k-1$&$\ldots$&$2m+2$\\\hline&$2k+1$&$2k$&$\ldots$&$2m+3$\\\hline&&$2k+1$&$2k$&$\ldots$\\\hline$\ldots$&$\ldots$&$\ldots$&$2k+1$&$2k$\\\hline&$\ldots$&$\ldots$&$\ldots$&$2k+1$\\\hline
\end{tabular}\\




 Now, arranging the two tableau and the idempotent and commutative latin square of order $k+1$ in the below fashion, would give us the color matrix desired, with $2k+1$ colors. The portion $C$ is the portion of the latin square $C'$ which fits in the color matrix. 

\begin{tabular}{|c|c|c|c|c|c|}\hline
$C$&\diagbox{$B$}{}&&&&\diagbox{}{$A$}\\\hline
\diagbox{}{$B^T$\\}&C&\diagbox{$A^T$\\}{}&&&\\\hline
&\diagbox{}{$A$}&C&\diagbox{$B$}{}&&\\\hline
&&\diagbox{}{$B^T$\\}&C&\diagbox{$A^T$\\}{}&\\\hline
&&&\diagbox{}{$A$}&C&\diagbox{$B$}{}\\\hline
\diagbox{$A^T$\\}{}&&&&\diagbox{}{$B^T$\\}&$C$\\\hline\end{tabular}\\
\\
In case $i=1$, the entries of $C'$ are written wholly, so that the tableau $A'$ and $B'$ are equal to the tableau $A$ and $B$. In case $i>1$, the tableau $A'$ and $B'$ could be modified to accommodate the missed numbers in the color matrix,  which are deleted from the commutative idempotent latin square $C'$ That is, the portion of the $k+i$ latin square starting from the $(k+2)^{nd}$ position in the first row, $(k+3)^{rd}$  position in the second row and so on, is cut and juxtaposed on the tableau $A'$ and $B'$ to give us the tableau $A$ and $B$. In particular, if $i=k+1$, the tableau $A'$ and $B'$ are wholly replaced with the portions deleted from the latin square $C'$. The portion deleted from the latin square, in case $i>1$ would be $D'$ and its transpose, where $D$ given by:\\
\\
\begin{tabular}{|c|c|c|c|c|}\hline$e_{1,k+2}$&$e_{1,k+3}$&$\ldots$&$\ldots$&$e_{1,k+i}$\\\hline&$e_{2,k+3}$&$e_{2,k+4}$&$\ldots$&$e_{2,k+i}$\\\hline&&$\ldots$&$\ldots$&$e_{3,k+i}$\\\hline&&&$\ldots$&$e_{4,k+i}$\\\hline&&&&$e_{k+i,k+i}$\\\hline\end{tabular},\\
\\

where $e_{ij}$ denote the entries of the latin square $C'$.

Say, for example, for the color matrix of the graph $C_{20}^4$, since $20=2(4+1)=2(2m+1)$ for $m=2$, we take $C'$ to be\\
\begin{tabular}{|c|c|c|c|c|}\hline1&4&2&5&3\\\hline4&2&5&3&1\\\hline2&5&3&1&4\\\hline5&3&1&4&2\\\hline3&1&4&2&5\\\hline\end{tabular}\\
\\
$A$ to be \\
\\
\begin{tabular}{|c|c|c|c|}\hline9&8&7&6\\\hline&9&8&7\\\hline&&9&8\\\hline&&&9\\\hline\end{tabular}\\
\\
and $B$ to be\\
\\
\begin{tabular}{|c|c|c|c|}\hline6&&&\\\hline7&6&&\\\hline8&7&6&\\\hline9&8&7&6\\\hline\end{tabular}\\
\\
Therefore, the color matrix desired is:\\
\\
\begin{tabular}{|c|c|c|c|c|c|c|c|c|c|c|c|c|c|c|c|c|c|c|c|}\hline1&4&2&5&3&&&&&&&&&&&&9&8&7&6\\\hline4&2&5&3&1&6&&&&&&&&&&&&9&8&7\\\hline2&5&3&1&4&7&6&&&&&&&&&&&&9&8\\\hline5&3&1&4&2&8&7&6&&&&&&&&&&&&9\\\hline3&1&4&2&5&9&8&7&6&&&&&&&&&&&\\\hline&6&7&8&9&1&4&2&5&3&&&&&&&&&&\\\hline&&6&7&8&4&2&5&3&1&9&&&&&&&&&\\\hline&&&6&7&2&5&3&1&4&8&9&&&&&&&&\\\hline&&&&6&5&3&1&4&2&7&8&9&&&&&&&\\\hline&&&&&3&1&4&2&5&6&7&8&9&&&&&&\\\hline&&&&&&9&8&7&6&1&4&2&5&3&&&&&\\\hline&&&&&&&9&8&7&4&2&5&3&1&6&&&&\\\hline&&&&&&&&9&8&2&5&3&1&4&7&6&&&\\\hline&&&&&&&&&9&5&3&1&4&2&8&7&6&&\\\hline&&&&&&&&&&3&1&4&2&5&9&8&7&6&\\\hline&&&&&&&&&&&6&7&8&9&1&4&2&5&3\\\hline9&&&&&&&&&&&&6&7&8&4&2&5&3&1\\\hline8&9&&&&&&&&&&&&6&7&2&5&3&1&4\\\hline7&8&9&&&&&&&&&&&&6&5&3&1&4&2\\\hline6&7&8&9&&&&&&&&&&&&3&1&4&2&5\\\hline\end{tabular}\\
\\

For the case of the color matrix of the graph $C_{18}^5$, since $18=2(8+1)=2(2m+1)$ for $m=4$ we have $C'$ to be \\
\\
\begin{tabular}{|c|c|c|c|c|c|c|c|c|}\hline1&6&2&7&3&8&4&9&5\\\hline6&2&7&3&8&4&9&5&1\\\hline2&7&3&8&4&9&5&1&6\\\hline7&3&8&4&9&5&1&6&2\\\hline3&8&4&9&5&1&6&2&7\\\hline8&4&9&5&1&6&2&7&3\\\hline4&9&5&1&6&2&7&3&8\\\hline9&5&1&6&2&7&3&8&4\\\hline5&1&6&2&7&3&8&4&9\\\hline\end{tabular}\\
\\
$A$ to be\\
\\
\begin{tabular}{|c|c|c|c|c|}\hline11&10&4&9&5\\\hline&11&10&5&1\\\hline&&11&10&6\\\hline&&&11&10\\\hline&&&&11\\\hline\end{tabular}\\
\\
and $B$ to be\\
\\
\begin{tabular}{|c|c|c|c|c|}\hline10&&&&\\\hline11&10&&&\\\hline4&11&10&&\\\hline9&5&11&10&\\\hline5&1&6&11&10\\\hline\end{tabular}\\
\\
Therefore, the color matrix would be\\
\\
\begin{tabular}{|c|c|c|c|c|c|c|c|c|c|c|c|c|c|c|c|c|c|}\hline1&6&2&7&3&8&&&&&&&&11&10&4&9&5\\\hline6&2&7&3&8&4&9&&&&&&&&11&10&5&1\\\hline2&7&3&8&4&9&5&1&&&&&&&&11&10&6\\\hline7&3&8&4&9&5&1&6&2&&&&&&&&11&10\\\hline3&8&4&9&5&1&6&2&7&10&&&&&&&&11\\\hline8&4&9&5&1&6&2&7&3&11&10&&&&&&&\\\hline&9&5&1&6&2&7&3&8&4&11&10&&&&&&\\\hline&&1&6&2&7&3&8&4&9&5&11&10&&&&&\\\hline&&&2&7&3&8&4&9&5&1&6&11&10&&&&\\\hline&&&&10&11&4&9&5&1&6&2&7&3&8&&&\\\hline&&&&&10&11&5&1&6&2&7&3&8&4&9&&\\\hline&&&&&&10&11&6&2&7&3&8&4&9&5&1&\\\hline&&&&&&&10&11&7&3&8&4&9&5&1&6&2\\\hline11&&&&&&&&10&3&8&4&9&5&1&6&2&7\\\hline10&11&&&&&&&&8&4&9&5&1&6&2&7&3\\\hline4&10&11&&&&&&&&9&5&1&6&2&7&3&8\\\hline9&5&10&11&&&&&&&&1&6&2&7&3&8&4\\\hline5&1&6&10&11&&&&&&&&2&7&3&8&4&9\\\hline\end{tabular}\\
\\
 \end{proof}

\begin{thm}
Let $G$ be a power of cycle graph $C_n^k$ with $n=s(2m+1)\pm1$, with $s$ being even and $m$ being positive integers $\frac{k}{2}<m<k$ and $k=2m+1-i$. Then the graph $C_n^k$ is total colorable with $2k+2$ colors. In fact, $\chi''(C_n^k)\leq 2k+3$.
\end{thm}
\begin{proof}
Case 1: $n=s(2m+1)-1$\\
From the previous theorem, it is clear that $C_{n+1}^k$ is $2k+1$ colorable. It remains to show that this coloring could be extended to a $2k+2$ coloring for $C_n^k$. The extension is made possible by deleting  last row and last coloumn of the color matrix of $C_{n+1}^k$ and appropriately adding a new color in the lower and upper $(n-k+1)^{th}$ subdiagonal of the color matrix so as to obtain the desired color matrix with $2k+2$ colors.\\
This method of extension applies to any powers of cycle graph of any odd order. Since it is already proved that the even order powers of cycles satisfy a type II coloring, therefore, by extension, since we require only one extra color, the total chromatic number of the graph is $2k+3$.\\

For example, the color matrix below for $C_{14}^3$\\
\\
\begin{tabular}{|c|c|c|c|c|c|c|c|c|c|c|c|c|c|}
\hline1&5&2&6&&&&&&&&3&7&4\\\hline5&2&6&3&7&&&&&&&&4&1\\\hline2&6&3&7&4&1&&&&&&&&5\\\hline6&3&7&4&1&5&2&&&&&&&\\\hline&7&4&1&5&2&6&3&&&&&&\\\hline&&1&5&2&6&3&7&4&&&&&\\\hline&&&2&6&3&7&4&1&5&&&&\\\hline&&&&3&7&4&1&5&2&6&&&\\\hline&&&&&4&1&5&2&6&3&7&&\\\hline&&&&&&5&2&6&3&7&4&1&\\\hline&&&&&&&6&3&7&4&1&5&2\\\hline3&&&&&&&&7&4&1&5&2&6\\\hline7&4&&&&&&&&1&5&2&6&3\\\hline4&1&5&&&&&&&&2&6&3&7\\\hline
\end{tabular}\\
\\

is modified to that of $C_{13}^3$ by deletion of one row and column and addition of the appropriate color to:
\\
\begin{tabular}{|c|c|c|c|c|c|c|c|c|c|c|c|c|}
\hline1&5&2&6&&&&&&&\textcircled{8}&3&7\\\hline5&2&6&3&7&&&&&&&\textcircled{8}&4\\\hline2&6&3&7&4&1&&&&&&&\textcircled{8}\\\hline6&3&7&4&1&5&2&&&&&&\\\hline&7&4&1&5&2&6&3&&&&&\\\hline&&1&5&2&6&3&7&4&&&&\\\hline&&&2&6&3&7&4&1&5&&&\\\hline&&&&3&7&4&1&5&2&6&&\\\hline&&&&&4&1&5&2&6&3&7&\\\hline&&&&&&5&2&6&3&7&4&1\\\hline\textcircled{8}&&&&&&&6&3&7&4&1&5\\\hline3&\textcircled{8}&&&&&&&7&4&1&5&2\\\hline7&4&\textcircled{8}&&&&&&&1&5&2&6\\\hline
\end{tabular}\\
\\

The encircled number is the added color to the  color matrix.\\

\noindent Case 2:$n=s(2m+1)+1$.\\
Here, we add a  row and a column at the last and then add entries in the last row and last column which are the $k$ deleted entries each from the lower and upper $k^{th}$ and $(n-k+1)^{th}$ subdiagonal of the original color matrix of $C_{n-1}^k$. In place of the entries deleted in the $k^{th}$ subdiagonal, we add a new color. The new vertex is also given the new color.  For example, the color matrix can be modified by adding one row and column to give the color matrix of $C_{15}^3$ as follows:
\\
\begin{tabular}{|c|c|c|c|c|c|c|c|c|c|c|c|c|c|c|}
\hline1&5&2&\textcircled{8}&&&&&&&&\textcircled{}&7&4&\textcircled{6}\\\hline5&2&6&3&\textcircled{8}&&&&&&&&\textcircled{}&1&\textcircled{7}\\\hline2&6&3&7&4&\textcircled{8}&&&&&&&&\textcircled{}&\textcircled{1}\\\hline\textcircled{8}&3&7&4&1&5&2&&&&&&&&\\\hline&\textcircled{8}&4&1&5&2&6&3&&&&&&&\\\hline&&\textcircled{8}&5&2&6&3&7&4&&&&&&\\\hline&&&2&6&3&7&4&1&5&&&&&\\\hline&&&&3&7&4&1&5&2&6&&&&\\\hline&&&&&4&1&5&2&6&3&7&&&\\\hline&&&&&&5&2&6&3&7&4&1&&\\\hline&&&&&&&6&3&7&4&1&5&2&\\\hline\textcircled{}&&&&&&&&7&4&1&5&2&6&\textcircled{3}\\\hline7&\textcircled{}&&&&&&&&1&5&2&6&3&\textcircled{4}\\\hline4&1&\textcircled{}&&&&&&&&2&6&3&7&\textcircled{5}\\\hline\textcircled{6}&\textcircled{7}&\textcircled{1}&&&&&&&&&\textcircled{3}&\textcircled{4}&\textcircled{5}&\textcircled{8}\\\hline
\end{tabular}\\
 \\
The encircled numbers show the changes from the color matrix of $C_{14}^3$
\end{proof}

\begin{thm}
Let $G$ be the cayley graph of a nilpotent group  of  even order $n$ with maximum degree $\Delta (G)\ge\frac{n}{2}$ and the generating set $S$  not containing an element of order two. If $G$ is total colorable with $\chi''(G)$ colors and if $G'$ is the cayley graph of nilpotent group of even order graph with maximum degree $\Delta(G')$, $\Delta(G) \le \Delta(G') \le n-2$ formed by with generating set $S'=S\cup S''$ such that $S''$ which generates the whole group and also does not have an order two element,  then graph $G'$ is total colorable with $\chi''(G)+(\Delta(G')-\Delta(G))$ colors. In particular, if $G$ is type I (type II), then $G'$ is also type I (type II).
\end{thm}
\begin{proof}
Let $s$ be an element of order two in the graph (which is guaranteed by Cauchy's theorem). Since the vertices of the graph $g_i,\ \ i\in\{1,2,\ldots,n\}$ can always be arranged in $\frac{n}{2}$ independent color classes as $[g_1, g_1s],[g_2, g_2s],\ldots,[g_{\frac{n}{2}}, g_{\frac{n}{2}}s]$,  this gives us a $\frac{n}{2}$ coloring of the vertices. Therefore, we need to take care only of the edge coloring of the graph $G'-G$ in order to give a total coloring of $G'$ from the existing total coloring of $G$. Now, since $G'-G$ is $1-$ factorizable \cite{STO}, therefore, we only need $\Delta'-\Delta$ extra colors, thereby giving a total coloring of $\chi''+(\Delta'-\Delta)$. Thus, if $G$ were a type I( type II) graph, then $G'$ also would be type I( type II).
\end{proof}

For a positive integer $n>1$ the \textit{unitary Cayley graph} $X_n=Cay(Z_n, U_n)$ is defined by the additive group of
    the ring $Z_n$ of integer modulo $n$ and the multiplicative group $U_n$ of its units. If we represent the elements of
    $Z_n$ by the integers 0,1,...$n-1$, then it is well known that \begin{center} $U_n=\{a\in Z_n: gcd(a,n)=1\}$.
    \end{center} So $X_n$ has vertex set $V(X_n)=X_n=\{0,1,2,...,n-1\}$ and edge set \begin{center} $E(X_n)=\{(a,b): a,b \in
    Z_n, gcd(a-b,n)=1\}$. \end{center} Boggess et al. ~\cite{BHJR} studied the structure of unitary Cayley graphs. They
    have also discussed chromatic number, vertex and edge connectivity, planarity and crossing number. Klotz and Sander
    ~\cite{ws07} have determined the clique number, the independence number and the diameter. They have given a necessary
    and sufficient condition for the perfectness  of $X_n$.

    The graph $X_n$ is regular of degree $U_n=\varphi(n)$ denotes the Euler function. Let the prime factorization of $n$ be
    $p_1^{\alpha_1} p_2^{\alpha_2}...p_t^{\alpha_t}$ where $p_1<p_2<...<p_t$.  If $n=p$ is a prime number, then $X_n=K_p$ is
    the complete graph on $p$ vertices. If $n=p^\alpha$ is a prime power then $X_n$ is a complete $p$-partite graph. In the
    following theorem we prove TCC holds for unitary Cayley graphs.

    \begin{thm} A unitary Cayley graph $X_n$ is $(\Delta(X_n)+2)$ - total
    colorable.  \end{thm}

    \begin{proof} We know that a unitary Cayley graph can be obtained from
        a balanced $r$ partite     graph by deleting some edges. Suppose
        $n=p$ is a prime number, then $X_n$ is the complete graph on $p$
        vertices.  Also, if $n=p^\alpha$, a prime power, then $X_n$ is a
        complete $p$-partite graph and TCC holds for these two graphs
        ~\cite{YAP}.

        When $n=2k, k \in N$, then the unitary Cayley graph is a bipartite
        graph and any bipartite graph is total colorable.

        Suppose $n \not \equiv 0 \mod 2 $. As $p_1$ is the smallest prime,
        $kp_1,kp_1+1,...,(k+1)p_1-1,$ where $k=0,1,2,...,\frac{n}{p_1}-1$
        induces $\frac{n}{p_1}$  vertex disjoint cliques each of order $p_1$.
        Since $p_1$ is odd, we can color all the elements of these
        $\frac{n}{p_1}$ cliques using $p_1$ colors ~\cite{79}. Now remove the
        edges of these cliques. The remaining graph is a
        $\varphi(n)-p_1+1$-regular graph where the vertices are already
        coloured. We  color the  edges of this resultant graph with
        $\varphi(n)-p_1+2$ colors. Thus we have used $\varphi(n)+2$ colors
    for the total coloring of $X_n$.  
    \end{proof}
    
\section{Mock Threshold Graphs and Odd Graphs}
    A graph is weakly chordal if neither the graph nor the complement of the graph has an  induced  cycle  on  five or  more
        vertices. A simple graph $G$ on $[n]= \{1, 2, ... , n\}$ is  threshold, if  $G$ can be built sequentially from
        the empty graph by adding vertices one at a time, where each new vertex is either isolated (nonadjacent to all the
        previous) or dominant (connected to all the previous).  A graph $G$ is said to be mock threshold if there is a vertex
        ordering $v_1, . . . , v_n$ such that for every $i \ (1 \leq i \leq n)$ the degree of $v_i$ in $G : {v_1, . . . , v_i}$ is
        0, 1, $i-2$, or $i-1$. Mock threshold graphs are a simple generalization of threshold graphs that, like threshold
        graphs, are perfect graphs. Mock threshold graphs are perfect and indeed weakly chordal but not necessarily chordal
        ~\cite{bsz18}. Similarly, the complement of a mock threshold graph is also mock threshold.

        In the following, we prove the TCC for Mock Threshold graphs.

        \noindent \textbf{Note:} A total coloring of $K_n$ can be constructed as follows: (This total coloring is due to Hinz and Parisse ~\cite{20})

        When $n$ is even, we first construct an edge coloring of $K_n$ and extend it. We denote $[n]_0=\{0,1,2,...,n-1\} $. For $k\in [n]_0$,
        let $\tau_k$ be the transposition of $k$ and $n-1$ on $[n]_0$. For even $n$, $c_n(i,j)=(\tau_i(j)+\tau_j(i)+2)\mod
        (n+1)$, for $i,j \in [n]_0, i\neq j$, defines a $(n+1)$-edge coloring. In this coloring assignment line $k\in [n]_0$
        will have the missing colors $k$ and $(k+1)\mod n$. We color $c_n(i)=i$ for all $i\in [n]_0$.\vspace{0.3cm}

        When  $n$ is odd,  we use the same coloring of $K_{n-1}$. In the coloring assignment of $K_{n-1}$, still the color
        $(k+1) \mod n$ is missing in line $k\in [n]_0$. We use these colors to the edges incident with $n^{th}$ vertex and
        color $n$ to the $n^{th}$ vertex.

        \begin{thm} Total coloring conjecture holds for any   Mock threshold graph $G$.  \end{thm} \begin{proof} Consider the
            Mock threshold graph $G$ with vertex ordering  $v_1, v_2, ..., v_i,..., v_n$.

            We prove this theorem using mathematical induction on the induced subgraph $G[v_1,v_2,...,v_k]$.

            For $k\leq 4$, the maximum degree of all the induced subgraphs is less than or equal to 3. We know that a graph with
            maximum degree less than or equal to 3 satisfies TCC ~\cite{kos96}.\vspace{0.3cm}

            Let us assume that $G[v_1,v_2,...,v_k], k\geq 5$  satisfies TCC.\vspace{0.3cm}

            \noindent \textbf{Claim:} The graph $G[v_1,v_2,...,v_k,v_{k+1}]$ satisfies TCC.\vspace{0.3cm}

            The degree of the vertex $v_{k+1}$ in $G[v_1,v_2,...,v_{k+1}]$  can be $0, 1,k-1$ or $k$.

            \noindent Case-1: Suppose  $d(v_{k+1})=0$.

            In this case the vertex is $v_{k+1}$ is an isolated vertex.  By the induction assumption, $G[v_1,v_2,...,v_k,v_{k+1}]$
            satisfies TCC.

            \noindent Case-2: Suppose $d(v_{k+1})=1$.

            In this case, the vertex $v_{k+1}$ is adjacent to a vertex, say $v_i$, in $G[v_1, v_2, ..., v_k]$. Since
            $G[v_1,v_2,...,v_k]$ is total colorable graph with at most $\Delta(G[v_1,v_2,...,v_k])+2$ colors, at each vertex there
            will be at least one missing color.  We assign this missing color to the edge $(v_i,v_{k+1})$, and  for the vertex
            $v_{k+1}$, we assign a color of a vertex which is not adjacent to $v_{k+1}$ and not the color of $v_i$.  Therefore,
            $G[v_1, v_2, ..., v_{k+1}]$ satisfies TCC.

            \noindent Case-3: Suppose  $d(v_{k+1})=k-1$.

            Let us  assume that the vertex $v_{k+1}$ is not adjacent with $v_i$ and  also assume that \\ $\Delta(G[v_1, v_2, ...,
            v_{k+1}])=k-1$. We consider following two cases:

            \noindent Subcase-1:  $k$ is even.

            Since $k$ is even, $k+1$ is odd. Construct a complete graph induced by the vertices
            $v_1,v_2,...v_{i-1},v_{i+1},...,v_{k+1}$.  Now, color this complete graph using  colors in the set $\{0,1,..., n+1\}$
            as given in  the note.  In this coloring assignment there will be one missing color at each of the vertices and they
            are distinct. Now, color the edges  $(v_i,v_j), i\neq j, j=1,2,...,v_{k+1}$,  with the missing colors. Assign the
            color $n-1$ to the vertex $v_i$. To get a total coloring of $G[v_1,v_2,...,v_{k},v_{k+1}]$, we remove these edges and
            there is no change in the maximum degree.

            \noindent Subcase-2:  $k$ is odd.

            In this case $k+1$ is even, say $2p$. It is known that  a graph of order $2p$ with maximum degree $2p-2$ satisfies TCC
            (see ~\cite{hil90, chen92}).

            \noindent Case-4: Suppose $d(v_{k+1})=k$.

            The  maximum degree of  $G[v_1,v_2,...,v_k,v_{k+1}]$ is   $k$. Construct a complete graph  on the vertex set $\{v_1,
            v_2, ..., v_k,v_{k+1}\}$. We know that the complete graph satisfies TCC. After removing the added edges we get a
            total coloring of $G[v_1,v_2,...,v_k,v_{k+1}]$

            Hence, in all the cases, the mock threshold graph satisfies TCC.

    \end{proof}
    The Kneser graph $K(n,k)$ is the graph whose vertices correspond to the $k$-element subsets of a set of $n$ elements,
        and where two vertices are adjacent if and only if the two corresponding sets are disjoint. A vertex in the odd graph
        $O_n$  is a $(n - 1)$-element subset of a $(2n -1)$-element set. Two vertices are connected by an edge if and only if
        the corresponding subsets are disjoint. Note that the odd graphs are particular case of Kneser graphs.

        \begin{thm} Odd graph $O_n$ satisfies TCC.
        \end{thm} \begin{proof}

            Consider a $2n-1$ element set $X$. Let $O_n$ be an odd graph defined from the subsets of $X$. Let $x$ be any element of
            $X$. Then, among the vertices of $O_n$, exactly $ \binom {2n-2} {n-2}$ vertices correspond to sets that contain $x$.
            Because all these sets contain $x$, they are not disjoint, and form an independent set of $O_n$. That is, $O_n$ has
            $2n - 1$ different independent sets of size $ \binom {2n-2} {n-2}$.  Further, every maximum independent set must
            have this form, so, $O_n$ has exactly $2n - 1$ maximum independent sets.

            If $I$ is a maximum independent set, formed by the sets that contain $x$, then the complement of $I$ is the set of
            vertices that do not contain $x$. This complementary set induces a matching in $G$. Each vertex of the independent
            set is adjacent to $n$ vertices of the matching, and each vertex of the matching is adjacent to $n - 1$ vertices of
            the independent set ~\cite{god80}.

            Based on the decomposition, we give a total coloring of $O_n$ in the following way:

            Assign $n$ colors to the edges between the vertices in the maximum independent set $I$ and the vertices in the matching.
            Color the edges in matching and the vertices in $I$  with a new color. Color one set of vertices in the matching
            with another new color and the second set of vertices with the missing colors at these vertices.  This will give a
            total coloring of $O_n$ using at most $n+2=\Delta(O_n)+2$ colors.

        \end{proof}


\noindent \textbf{Acknowledgements:} This research work was supported by SERB, India (No.SERB: EMR/2017/001869).
 

\end{document}